\documentclass[10pt]{amsart}
\usepackage{amsmath}
\usepackage{amsthm}
\usepackage{amsfonts}
\usepackage{amssymb}
\usepackage{hyperref}

\theoremstyle{plain}
\newtheorem{theorem}{Theorem}[section]
\newtheorem{lemma}[theorem]{Lemma}

\numberwithin{equation}{section}

\newcommand{\nequiv}{\not\equiv}

\let\abs=\envert

\let\lm=\lambda
\newcommand{\floor}[1]{\left\lfloor#1\right\rfloor}

\newcommand{\mathmod}[1]{\ \left(\mathrm{mod}\ #1\right)}

\DeclareMathOperator{\Ei}{Ei}
\DeclareMathOperator{\LCM}{LCM}

\begin{document}
\title[Nonexistence of odd multiperfect numbers of a certain form]
{Explicit sieve estimates and nonexistence of odd multiperfect numbers of a certain form}
\author{Tomohiro Yamada}
\address{Institute for Promotion of Higher Education, Kobe University,
657-0011, 1-2-1, Tsurukabuto, Nada, Kobe, Hyogo, Japan}
\email{tyamada1093@gmail.com}

\date{}

\begin{abstract}
We prove explicit asymptotic formulae for some functions used in sieve methods
and show that there exists no odd multiperfect number of abundancy four
whose squared part is cubefree.
\end{abstract}

\maketitle

\section{Introduction}\label{intro}

A positive integer $N$ is called perfect if the sum of divisors of $N$ except $N$ itself is equal to $N$.
In other words, a perfect number is a positive integer $N$ satisfying $\sigma(N)=2N$,
where $\sigma(N)$ denotes the sum of divisors of $N$ including $N$ itself as usual.
One of the oldest problem in mathematics is whether or not an odd perfect number exists.
Euler showed that an odd perfect number must be of the form
\begin{equation}\label{eq11}
N=q_r^{\alpha} q_1^{2\beta_1}\cdots q_{r-1}^{2\beta_{r-1}}
\end{equation}
where $q_1, \ldots, q_r$ are distinct odd primes
and $\alpha, \beta_1, \ldots, \beta_{r-1}$ are positive integers with $q_r\equiv \alpha\equiv 1\mathmod{4}$.
We wrote the special prime for $q_r$ instead of usually used $p$
since we would like to write $p$ for a general prime.
Nielsen proved that $N<2^{4^r}$ in \cite{Nie1} and then $N<2^{(2^r-1)^2}$ in \cite{Nie2}.

However, we do not know a proof (or disproof) of the nonexistence of odd perfect numbers
even of the special form
\begin{equation}\label{eq12}
N=q_r^{\alpha} (q_1 q_2 \cdots q_{r-1})^{2\beta},
\end{equation}
where $q_1, \ldots, q_r$ are distinct odd primes again and $\alpha$ and $\beta$
are positive integers,
although McDaniel and Hagis conjectured that there exists no such one in \cite{MDH}.
A series of papers studying odd perfect numbers of this form
was started by Steuerwald \cite{St}, who showed that $\beta\neq 1$.
Four years later, Kanold \cite{Kan1} made a significant progress by proving that
i) if $N$ in the form \eqref{eq11} is an odd perfect number and
an integer $t$ divides $2\beta_i+1$ for $i=1, 2, \ldots , r-1$, then $t^4$ divides $N$,
ii) if all $2\beta_i+1$ are powers of the common prime $\ell$, then $\ell=q_i$ for some $i$,
and iii) $\beta=2$ in \eqref{eq12} can never occur.
Gathering more recent results such as \cite{CW}, \cite{FNO}, \cite{HMD}, \cite{Mc}, and \cite{MDH},
we know that if $N$ is an odd perfect number of the form \eqref{eq12},
then $\beta\geq 9$, $\beta\not\equiv 1\mathmod{3}, \beta\not\equiv 2\mathmod{5}$,
and $\beta$ cannot take some other values such as $11$, $14$, $18$, or $24$.

The author proved that $r\leq 4\beta^2+2\beta+3$ and $N<2^{4^{4\beta^2+2\beta+3}}$
for an odd perfect number $N$ in the form \eqref{eq12} in \cite{Ymd1}.
Later, the author replaced the upper bound for $r$ by
$2\beta^2+8\beta+3$ in \cite{Ymd2} and then $2\beta^2+6\beta+3$ in \cite{Ymd3}.

Moreover, it is also unknown whether or not there exists an odd multiperfect number,
an integer dividing the sum of its divisors.
We call a positive integer $N$ to be $k$-perfect if $\sigma(N)=kN$.
Ordinary perfect numbers are $2$-perfect numbers
and multiperfect numbers are $k$-perfect numbers for some integer $k$.

Kanold \cite{Kan2} proved that if $N=p_1^{\alpha_1} \cdots p_r^{\alpha_r}$
is an odd $k$-perfect number for an integer $k$
(we note that in this paper, Kanold uses the notion ``eine $s-1$-fach vollkommene Zahl''
when $\sigma(N)=sN$, instead of $\sigma(N)=(s-1)N$) and
$t$ divides $\alpha_i+1$ for all $i$'s except at most one, then $t$ divides $kN$.
Moreover, if additionally $t$ or $k$ is a power of two, then $t^3$ divides $N$.
Broughan and Zhou \cite{BZ} studied odd $4$-perfect numbers and
proved that an odd $4$-perfect number must be either of the forms
i) $p_1^{\alpha_1} p_2^{\alpha_2} q_1^{2\beta_1} \cdots q_{r-2}^{2\beta_{r-2}}$
with $p_i\equiv \alpha_i\equiv 1\mathmod{4}$ for $i=1, 2$ or
ii) $q_r^\alpha q_1^{2\beta_1} \cdots q_{r-1}^{2\beta_{r-1}}$
with a) $q_r\equiv 1\mathmod{4}$ and $\alpha\equiv 3\mathmod{8}$ or
b) $q_r\equiv 3\mathmod{8}$ and $\alpha\equiv 1\mathmod{4}$,
which is an analogue of Euler's result for odd $4$-perfect numbers.
They also claimed that there exists no odd $4$-perfect number whose squared part is cubefree.
But they overlooked numbers of the form $3^\alpha (q_1 q_2 \cdots q_{r-1})^2$,
which is our main concern of this paper.

Now we would like to extend our results in \cite{Ymd1}, \cite{Ymd2}, and \cite{Ymd3}
for odd perfect numbers in the form \eqref{eq12}
into odd $k$-perfect numbers by showing that there are only finitely many odd $k$-perfect numbers
in the form \eqref{eq12} for any fixed $k$ and $\beta$.
But our method cannot be extended into odd $k$-perfect numbers in all cases.
The starting point of our argument in papers mentioned above is the fact that
if $\ell=2\beta+1$ is prime, then we must have $\ell=q_i$ for some $i\leq r-1$.
This is no longer true in general for odd $k$-perfect numbers.

It is not difficult to show that if an integer $N$ of the form \eqref{eq12} is $k$-perfect
and, in the case $2\beta+1=\ell^g$ with $\ell$ a prime and $g\geq 1$ an integer,
$k=k_1 \ell^t$ for integers $t\geq 0$ and $k_1\nequiv 0, 1\mathmod{\ell}$,
then we can take a prime divisor of $2\beta+1$ other than $q_r$.
Using methods employed by \cite{Ymd1}, \cite{Ymd2}, and \cite{Ymd3}, it follows that 
\begin{theorem}\label{th0}
\[r<\left(2\beta+\frac{\log\log \beta+1.24351}{\log 3}\right)(\beta+3)+4\]
for $\beta\geq 8$ and $r\leq 14$, $30$, $56$, $90$, $132$, $182$, and $240$
for $\beta=1$, $2$, $3$, $4$, $5$, $6$, and $7$  respectively.
\end{theorem}

However, this argument does not work if $2\beta+1=q_r^g$ and $k=k_1 q_r^t$
for some integers $g\geq 1$, $t\geq 0$, and $k_1\equiv 1\mathmod{q_r}$.
The simplest would be the case $(k, q_r, \beta)=(4, 3, 1)$.
Thus we encounter the possibility of $4$-perfect numbers of the form
$3^\alpha (q_1 q_2 \cdots q_{r-1})^2$ again.

The main purpose of this paper is to prove the impossibility of this case.

\begin{theorem}\label{th1}
Under the notation described above,
no integer of the form \\ $N=3^\alpha (q_1 q_2 \cdots q_{r-1})^2$ with $3, q_1, q_2, \ldots, q_{r-1}$
distinct odd primes is a $4$-perfect number.
\end{theorem}

Combined with Theorem 2.5 of Broughan and Zhou \cite{BZ}, we see that
there exists no odd $4$-perfect number whose squared part is cubefree.

We draw the outline of our proof.
Clearly, if $N$ in the above form is $4$-perfect, $\alpha$ must be odd and $4$ divides $\sigma(3^\alpha)$.
Hence, $N=\sigma(N)/4=(\sigma(3^\alpha)/4) \prod_{i=1}^{r-1} \sigma(q_i^2)$.
From this, we can easily see that if a prime $p>3$ divides three of $\sigma(q_i^2)$'s,
then $p^3\mid N$, which is a contradiction.
Hence, for each prime $p>3$, there exist at most two primes $q_i$
such that $p$ divides $\sigma(q_i^2)=q_i^2+q_i+1$.
It follows that for any real $w$, with at most $2\pi(w)$ exceptions, integers $q_i^2+q_i+1$
has no prime factor $\equiv 1\mathmod{3}$ below $w$
(as is well known, a prime factor $q_i^2+q_i+1$ must be three or $\equiv 1\mathmod{3}$
by Theorem 94 of \cite{Nag}).
The set of such integers can be studied by sieve methods and
we shall show that some $q_i$ must be smaller than a constant
and then no prime smaller than a constant can be the smallest among $q_i$'s,
which is a contradiction.

An attempt to prove the nonexistence of odd perfect numbers with restricted exponents using sieve methods
has started with the author's preprint \cite{Ymd05}, which proved that if an odd integer of the form
$N=p_1^{\alpha_1}\cdots p_s^{\alpha_s} q_1^{\beta_1}\cdots q_t^{\beta_t}$ satisfies $\sigma(N)/N=n/d$,
with $n$ and $d$ positive integers and $\beta_1+1, \ldots, \beta_t+1$ in a given finite set $\mathcal{P}$,
then $N$ must have a prime divisor smaller than a constant $C$
effectively computable in terms of $s$, $n$, and $\mathcal{P}$.
While Selberg's small sieve was used in \cite{Ymd05}, Fletcher, Nielsen, and Ochem \cite{FNO} used
the large sieve to prove that if $N=p_1^{\alpha_1}\cdots p_s^{\alpha_s} q_1^{\beta_1}\cdots q_t^{\beta_t}$
satisfies $\sigma(N)/N=n/d$ and for each $i$, $\beta_i+1$ has a prime factor which belongs to
a finite set $\mathcal{P}$ of primes, then $N$ has a prime divisor smaller than
an effectively computable constant $C$, depending only on $n$, $s$ and $\mathcal{P}$.
Moreover, they proved that the smallest prime factor of an odd perfect number
$N$ satisfying the above condition with $\mathcal{P}=\{3, 5\}$
lies between $10^8$ and $10^{1000}$, improving results in \cite{Coh} and \cite{Ymd05}.

Now we would be able to prove the nonexistence of an odd multiperfect number of the form
$N=p_1^{\alpha_1}\cdots p_s^{\alpha_s} q_1^{\beta_1}\cdots q_t^{\beta_t}$
with $2\beta_i +1$ divisible by a prime in a given finite set $\mathcal{P}$
by confirming that no prime $p$ below the resulted constant can be the smallest prime factor.
For each prime $p$ outside $\mathcal{P}$, we could show it if we could deduce from $p\mid N$ that
some prime factor $q<p$ must divide $N$, which is clearly a contradiction.
Then, all that remains would be to derive a contradiction from each prime in $\mathcal{P}$.
This can be done by assuming that a given prime $p$ divides $N$,
taking a prime factor $p^\prime$ of $\sigma(p^2)$ other than three, which must divide $N$,
and iterating this process until we obtain a prime smaller than $p$.

However, upper bounds obtained by papers mentioned above are extremely large,
which prevent us from checking all primes below the given bound.

In general, for a set $A$ of integers contained in an interval of length at most $X$
and a system $\Omega$ of congruent classes consisting of
$\rho(p)$ congruent classes for each prime $p\leq w$,
the number of integers in $A$ which does not belong to any congruent class in $\Omega$
is at most $X/M_g(w)+E$ by Selberg's sieve or the large sieve,
where $M_g(w)=\sum_{n\leq w} g(w)$ and $g(w)$ is the completely multiplicative function
supported only on squarefree integers with $g(p)=\rho(p)/(p-\rho(p))$ for each prime $p$
and $E$ is a remainder term of order $O(w^{2+\epsilon})$ for any $\epsilon>0$.

In \cite{Ymd05}, we used Selberg's sieve to obtain $E=O(w^2\log^\alpha w)$ for some constant $\alpha$
while in \cite{FNO} as well as the older preprint version of this paper {\tt arXiv:2103.6936v2},
the large sieve was used to obtain $E\leq w^2$.
Now we use Selberg's sieve again.
However, we examine the error term of Selberg's sieve inequality more carefully
to obtain the remainder term of order $O(w^2/M_g^2(w))$, which is stronger than the large sieve estimate.
(it may be notable that Selberg's sieve with careful examination of the error term sometimes gives
a stronger upper bound than the large sieve).

Fletcher, Nielsen, and Ochem \cite{FNO} suggest that
methods used in Theorem 2.2.2 of \cite{Gre} and Theorem 1.1 of \cite{IK}
allow us to obtain an asymptotic and effective estimate for $M_g(w)$.
It should be noted that these theorems are essentially due to Wirsing \cite{Wir}.
Based on Wirsing's method, Moree \cite{Mor} gives an argument to obtain effective estimates of $M_g(w)$
for a wide class of multiplicative functions $g(n)$.
In the older preprint version mentioned above,
we used Wirsing's method here to obtain an asymptotic explicit formula.
In this paper, we use a convolution method in \cite{RV} instead to obtain an estimate of the form
$\abs{M_g(x)-(C_0 \log^2 x+C_1\log x+C_2)}<4444x^{-1/3}$ for $x\geq 1$ with some constants
$C_0$, $C_1$, and $C_2$ in Theorem \ref{th21},
which is stronger than our estimate in the older version for large $x$.
A convolution method used in \cite{RV} is based on a classical argument of Ward \cite{War},
who obtained asymptotic formulae for some series involving Euler's function such as
$\sum_{n\leq x} \mu^2(n)/\varphi(n)$.
Extending a convolution method, Alterman \cite{Alt} and Ramar\'e and Akhilesh \cite{RA}
gave some generalized results for such a type of series.

So that, the major part of this paper is devoted to the derivation of an explicit asymptotic formula
for $M_g(x)$, which is achieved in the next section.
Together with estimates for $M_g(x)$ for small values of $x$,
we are able to obtain an upper bound for the smallest prime factor of $N$
below which we can check all primes.
PARI-GP commands used for our calculations and related tables are available from \cite{Ymd21}.

If our argument in the next section could be used to obtain explicit estimates
for sieves by primes in other arithmetic progressions, then
estimates obtained in \cite{FNO}, \cite{Ymd19}, and \cite{Ymd20} would be improved.
For example, applying the argument in the next section to primes $\equiv 1\mathmod{5}$,
the upper bound $10^{1000}$ for sixth-power free odd perfect numbers in \cite{FNO}
mentioned above may be considerably reduced.

\section{Selberg's sieve with explicit estimates}\label{LSEE}

We begin by describing Selberg's sieve method.

We write $F=G+O^*(H)$ to mean that $\abs{F-G}\leq H$.
We work on a set of integers $A$, a real number $X>0$, and subsets $A_p$ of $A$ given for each $p$.
Our purpose is to obtain an estimate for the number $S(A, z)=S(A, z, \Omega)$ of integers in $A$
that do not belong to $A_p$ for any prime $p\leq z$.

For squarefree integers $d$, let $A_d$ be the set of integers $A$ that belong to $A_p$ for each prime $p\mid d$
with $A_1=A$,
$\rho(d)$ be a multiplicative function supported over squarefree integers such that $0\leq \rho(p)<p$
for any $p$ and $r(d)=\# A_d-X\rho(d)/d$.
Let $g(m)$ be the multiplicative function supported only on squarefree integers $m$
defined by $g(p)=\rho(p)/(p-\rho(p))$ for each prime $p$ and
\[M_g(z)=\sum_{n\leq z}g(n).\]

Combining (1.1), (1.4), and (1.7) in Chapter 3 of \cite{HR}
(or Theorem 6.4 of \cite{IK}) gives the following estimate:
\begin{lemma}\label{lm0}
For any $w\geq 1$ we have
\begin{equation}
S(A, w)\leq \frac{X}{M_g(w)}+R(w),
\end{equation}
where
\begin{equation}
\lm(d)=\frac{d\mu(d)}{\rho(d)M_g(w)}\sum_{q\leq z, d\mid q}\mu^2(q)g(q)
\end{equation}
and
\begin{equation}
R(w)=\sum_{d_1, d_2\mid P(w)}\abs{\lm(d_1)\lm(d_2)r(\LCM[d_1, d_2])}.
\end{equation}
\end{lemma}

So that, our concern is the lower estimate for $M_g(x)$ as we explained in the Introduction.
We use an argument based on a traditional mathod of Ward involving convolutions of multiplicative functions.

In our situation, we take
$\rho(p)=0 (p=2, 3)$, $3 (p\equiv 1\mathmod{3})$, and $1 (p>3, p\equiv 2\mathmod{3})$.
Moreover, we write $\chi(n)=(n/3)$ for the primitive Dirichlet character modulo three.

Let $\tau(n)$ denote the number of divisors of $n$
and $f(m)$ be the arithmetic convolution
$f(m)=\sum_{d\mid m}\chi(d)\tau(m/d)$ of $\chi$ and $\tau$.

For brevity, we write $L=L(1, \chi)$, $L^\prime=L^\prime(1, \chi)$,
and $L^{\prime\prime}=L^{\prime\prime}(1, \chi)$.

\begin{lemma}\label{lm21}
For $z\geq 9995$, we have
\begin{equation}
\sum_{n\leq z}\frac{\tau(n)}{n}
=\frac{\log^2 z}{2}+2\gamma\log z+B_0+O^*\left(\frac{0.191(10\log z+9)}{z^{2/3}}\right),
\end{equation}
where $B_0=\gamma^2-2\gamma_1=0.47880961477507\cdots$ with
the first Stieltjes constant $\gamma_1=-0.0728158454836767\cdots$.
\end{lemma}

\begin{proof}
From Theorem 1.2 of \cite{BBR} we know that for all $w\geq 9995$,
$M_\tau(w)=w\log w+(2\gamma-1)w+\Delta(w)$
with $\abs{\Delta(w)}\leq 0.764 w^{1/3} \log w$.
Partial summation immediately gives that
\begin{equation}
\begin{split}
\sum_{n\leq z}\frac{\tau(n)}{n}= & ~ \frac{M_\tau(z)}{z}+\int_1^z\frac{M_\tau(t)}{t^2} dt \\
= & ~ \log w+(2\gamma-1)+\frac{\Delta(z)}{z}+\frac{\log^2 z}{2}
+(2\gamma-1)\log z+\int_1^z \frac{\Delta(t)}{t^2} dt.
\end{split}
\end{equation}
Hence, taking $B_0=2\gamma-1+\int_1^\infty \Delta(t)t^{-2}dt$, we have
\begin{equation}
\begin{split}
\sum_{n\leq z}\frac{\tau(n)}{n}
= & ~ \frac{\log^2 z}{2}+2\gamma\log z+B_0+\frac{\Delta(z)}{z}+\int_z^\infty \frac{\Delta(t)}{t^2}dt \\
= & ~ \frac{\log^2 z}{2}+2\gamma\log z+B_0+O^*\left(\frac{0.191(10\log z+9)}{z^{2/3}}\right)
\end{split}
\end{equation}
for $z\geq 9995$.
Finally, we see that $B_0=\gamma^2-2\gamma_1$ from Lemma 1 of \cite{RV},
which gives the asymptotic formula like the lemma with the error term $O^*(1.641z^{-1/3})$.

We note that in Corollary 2.2 of \cite{BBR} and Lemma 3.3 of \cite{Ram1},
the first Stieltjes constant term is erroneously given as $-\gamma_1$.
Correctly it is $-2\gamma_1$, as given in \cite{RV}.
\end{proof}

\begin{lemma}\label{lm22}
For $y\geq e^{21}$, we have
\begin{equation}\label{eq21}
\sum_{m\leq y}\frac{f(m)}{m}=
b_0 \log^2 y+b_1\log y+b_2+O^*\left(\frac{0.358045\log^2 y}{y^{2/5}}\right),
\end{equation}
where
\begin{equation}
b_0=\frac{L}{2}, b_1=2\gamma L+L^\prime, b_2=B_0 L+2\gamma L^\prime +\frac{L^{\prime\prime}}{2}.
\end{equation}
\end{lemma}

\begin{proof}
Assume that $y\geq e^{21}$ and
we put $\delta\geq 1$ to be a certain real number which will be chosen later
so that $y^{3/5}/\delta\geq 9995$ and $u=\delta y^{2/5}$.
Using Lemma \ref{lm21}, we have
\begin{equation}
\begin{split}
\sum_{m\leq y}\frac{f(m)}{m}= & ~ 
\sum_{k\leq u}\frac{\chi(k)}{k}\sum_{\ell\leq y/k}\frac{\tau(\ell)}{\ell}
+\sum_{\ell\leq y/u}\frac{\tau(\ell)}{\ell}\sum_{u<k\leq y/\ell}\frac{\chi(k)}{k} \\
= & ~ \sum_{k\leq u} \frac{\chi(k)}{k}\left(\frac{\log^2 (y/k)}{2}+2\gamma\log\frac{y}{k}
+B_0+E\left(\frac{y}{k}\right)\right)+E_1 \\
= & ~ \left(\frac{\log^2 y}{2}+2\gamma\log y+B_0\right)(L-E_{2, 0})+(2\gamma+\log y)(L^\prime -E_{2, 1}) \\
 & ~
+\frac{L^{\prime\prime}-E_{2, 2}}{2}+\sum_{k\leq u}\frac{\chi(k)}{k}E\left(\frac{y}{k}\right)+E_1 \\
= & ~ 
\left(\frac{\log^2 y}{2}+2\gamma\log y+B_0\right)L+(2\gamma+\log y)L^\prime+\frac{L^{\prime\prime}}{2}+E_1+E_2+E_3,
\end{split}
\end{equation}
where we put
\begin{equation}
\begin{split}
E_1= & ~ \sum_{\ell\leq y/u}\frac{\tau(\ell)}{\ell}\sum_{u<k\leq y/\ell}\frac{\chi(k)}{k}, \\
E_{2, r}= & ~ \sum_{k>u} \frac{\chi(k)\log^r k}{k} ~ (r=0, 1, 2),\\
E_2= & ~ -\left(\frac{\log^2 y}{2}+2\gamma\log y+B_0\right)E_{2, 0}-(2\gamma+\log y)E_{2, 1}-\frac{E_{2, 2}}{2}, \\
E_3= & ~ \sum_{k\leq u}\frac{\chi(k)}{k}E\left(\frac{y}{k}\right)
\end{split}
\end{equation}
and $E(z)$ is the error term in Lemma \ref{lm21}.

We can easily see that
\begin{equation}
\abs{\sum_{u<k\leq y/\ell}\frac{\chi(k)}{k}}<\frac{1}{\floor{u+1}}<\frac{1}{u}
\end{equation}
and, with the aid of Lemma \ref{lm21}, we have
\begin{equation}
\begin{split}
\abs{E_1}< & ~ \frac{1}{u}\sum_{\ell\leq y^{3/5}/\delta}\frac{\tau(\ell)}{\ell} \\
< & ~ \frac{1}{\delta y^{2/5}}
\left(\frac{9\log^2 y}{50}+\frac{(6\gamma-3\log\delta)\log y}{5}\right. \\
 & ~ \qquad \quad \left. +\frac{\log^2\delta}{2}-2\log\delta +B_0+\frac{0.191\delta^{2/3}(6\log y-10\log\delta +9)}{y^{2/5}}\right).
\end{split}
\end{equation}
Moreover, we observe that
\begin{equation}
\begin{split}
\sum_{k>u}\frac{\chi(k)\log^r k}{k}= ~ & 
\sum_{3\ell -1>u}\frac{\log^r (\ell -1)}{3\ell -1}-\frac{\log^r (3\ell +1)}{3\ell +1} \\
< ~ & \sum_{3\ell -1>u}\frac{2\log^r (3\ell -1)}{(3\ell -1)^2}
<2\sum_{m=0}^\infty\frac{\log^r (u+3m)}{(u+3m)^2}
\end{split}
\end{equation}
if $3m-2\leq u<3m-1$ for some integer $m$ and
\begin{equation}
\sum_{k>u}\frac{\chi(k)\log^r k}{k}<\sum_{m=0}^\infty\frac{\log^r (u+3m)}{(u+3m)^2}
\end{equation}
otherwise.
Using Euler-Maclaurin summation formula, we have
\begin{equation}
\sum_{m=0}^\infty\frac{\log^r(u+3m)}{(u+3m)^2}=
\int_0^\infty\frac{\log^r(u+3t)}{(u+3t)^2}dt+\frac{\log^r u}{2u^2}+O^*\left(\frac{\log^r u}{2u^2}\right)
\end{equation}
and
\begin{equation}
\abs{E_{2, r}}<\frac{2P_r(\log u)}{3u}+\frac{2\log^r u}{u^2},
\end{equation}
where $P_0(X)=1$, $P_1(X)=X+1$, and $P_2(X)=X^2+2X+2$.
Similarly, we have
\begin{equation}
\begin{split}
\sum_{u<k\leq y/\ell}\frac{\chi(k)\log^r k}{k}= ~ & 
\sum_{3\ell -1>u}\frac{\log^r (\ell -1)}{3\ell -1}-\frac{\log^r (3\ell +1)}{3\ell +1} \\
< ~ & \sum_{3\ell -1>u}\frac{2\log^r (3\ell -1)}{(3\ell -1)^2}
<2\sum_{m=0}^\infty\frac{\log^r (u+3m)}{(u+3m)^2}
\end{split}
\end{equation}
Finally, using Lemma \ref{lm21} again, we have
\begin{equation}\label{eq2e}
\begin{split}
\abs{E_3}< ~ & \frac{0.191}{y^{2/3}}\sum_{1\leq k\leq \delta y^{2/5}, 3\nmid k}\frac{10\log(y/k)+9}{k^{1/3}} \\
< ~ & \frac{0.191}{y^{2/3}}
\sum_{v=1}^2\sum_{0\leq k\leq (\delta y^{2/5}-v)/3}\frac{10\log(y/(3k+v))+9}{(3k+v)^{1/3}} \\
< ~ & \frac{0.191}{y^{2/3}} \sum_{v=1}^2 \left(\frac{10\log\frac{y}{v}+9}{v^{1/3}}
+\int_0^{\frac{\delta y^{2/5}-v}{3}}\frac{10\log\frac{y}{3t+v}+9}{(3t+v)^{1/3}} dt\right).
\end{split}
\end{equation}
Since
\begin{equation}
\begin{split}
\int_0^{\frac{\delta y^{2/5}-v}{3}} \frac{10\log\frac{y}{3t+v}+9}{(3t+v)^{1/3}}dt
= ~ & \int_v^{\delta y^{2/5}} \frac{10\log\frac{y}{u}+9}{3u^{1/3}}du \\
= ~ & \abs{u^{2/3}\left(5\log\frac{y}{u}+12\right)}_v^{\delta y^{2/5}},
\end{split}
\end{equation}
the last sum in \eqref{eq2e} is at most
\begin{equation}
\begin{split}
& 10\log y+2^{2/3}\times 5\log\frac{y}{2}+9\left(1+\frac{1}{2^{1/3}}\right)
+2\delta^{2/3}y^{4/15}\left(5\log\frac{y^{3/5}}{\delta}+12\right) \\
- & ((1+2^{2/3})(5\log y+12)-2^{2/3}\times 5\log 2) \\
& \quad =2\delta^{2/3}\left(5\log\frac{y^{3/5}}{\delta}+12\right)y^{4/15}+5\log y-\left(3+\frac{15}{2^{1/3}}\right)
\end{split}
\end{equation}
and therefore
\begin{equation}
\abs{E_3}<\frac{0.382\delta^{2/3}\left(5\log\frac{y^{3/5}}{\delta}+12\right)}{y^{2/5}}+\frac{0.191(5\log y-(3+15/2^{1/3}))}{y^{2/3}}.
\end{equation}

Hence, taking $\delta=12.5$, provided that $y\geq e^{21}$,
we confirm that $y^{3/5}/\delta\geq 9995$ and obtain $\abs{E_i}<u_i(\log y)^2 / y^{2/5}$ for each $i=1, 2, 3$ with
$u_1=0.01905447\cdots$, $u_2=0.11666012\cdots$, and $u_3=0.22233031\cdots$.
We see that $u_1+u_2+u_3<0.358045$ for $y\geq e^{21}$
to obtain \eqref{eq21} for $y\geq e^{21}$, proving the lemma.
\end{proof}

\begin{lemma}\label{lm23}
For $y>0$ and $0<\alpha<2/5$, we have
\begin{equation}\label{eq22}
\sum_{m\leq y}\frac{f(m)}{m}=b_0 \log^2 y+b_1\log y+b_2+O^*(c_\alpha y^{-\alpha})
\end{equation}
for some constant $c_\alpha$.
In particular, we can take $c_\alpha=38.9372$ for $\alpha=1/3$.
\end{lemma}

\begin{proof}
We see that Lemma \ref{lm22} implies \eqref{eq22} for $y\geq e^{21}$ since
\begin{equation}
\frac{0.358045\log^2 y}{y^{2/5}}<\frac{c_\alpha}{y^\alpha}.
\end{equation}
For $1\leq y\leq e^{21}$, calculation gives \eqref{eq22} (see two files of \cite{Ymd21}).

Finally, if $0\leq y<1$, then, taking $u=1/y$,
\begin{equation}
(b_0 \log^2 y+b_1\log y+b_2)y^\alpha=\frac{b_0\log^2 u-b_1\log u+b_2}{u^\alpha}\leq c_\alpha
\end{equation}
and therefore we have \eqref{eq22}.
\end{proof}

We define the multiplicative function $h(n)$ by
\begin{equation}
\begin{split}
& (h(p), h(p^2), h(p^3), h(p^4))\\
& = \begin{cases}
\left(-\frac{1}{2}, -\frac{1}{4}, \frac{1}{8}, 0\right) & (p=2), \\
\left(-\frac{2}{3}, \frac{1}{9}, 0, 0\right) & (p=3), \\
\left(\frac{9}{p(p-3)}, -\frac{6p+9}{p^2(p-3)}, \frac{8p+3}{p^3(p-3)}, -\frac{3}{p^3(p-3)}\right) & (p\equiv 1\mathmod{6}), \\
\left(\frac{1}{p(p-1)}, -\frac{2p-1}{p^2(p-1)}, -\frac{1}{p^3(p-1)}, \frac{1}{p^3(p-1)}\right) & (p\equiv 5\mathmod{6})
\end{cases}
\end{split}
\end{equation}
and $h(p^e)=0$ for $e\geq 5$.
Then we see that $g(n)=\sum_{d\mid n} h(n/d) f(d)/d$.

We put 
\begin{equation}
G(s)=\sum_{d=1}^\infty \frac{g(d)}{d^s}, ~ H(s)=\sum_{d=1}^\infty \frac{h(d)}{d^s}, ~
J(s)=\sum_{d=1}^\infty \frac{\abs{h(d)}}{d^s}.
\end{equation}
Then we can see that
\begin{equation}
G(s)=H(s)\sum_m \frac{f(m)}{m^{s+1}}=H(s)\zeta^2(s+1)L(s+1, \chi)
\end{equation}
Moreover, we observe that $H^{(r)}(-\alpha)=\sum_d h(d) d^\alpha \log^r d$
and $J^{(r)}(-\alpha)=$ \\ $\sum_d \abs{h(d)} d^\alpha \log^r d$.

We have the following approximate values.
\begin{lemma}\label{lm24}
We have
\begin{equation}
H(0)=0.181151826\cdots, H^\prime(0)=0.52500363\cdots, H^{\prime\prime}(0)=-0.121497\cdots,
\end{equation}
and $J(-1/3)<114.1231$.
\end{lemma}

\begin{proof}
It follows from Theorem 1 of \cite{BKLNW} that
$\abs{\theta(t)-t}<\epsilon_1$
for $t\geq e^{21}$ with $\epsilon_1=4.5673\times 10^{-5}$.
Moreover, it follows from Theorems 1.2 and 1.9 of \cite{BMOBR} that
$\abs{\theta(t; 3, 1)-t/2}<\epsilon_2$
for $t\geq e^{21}$ with $\epsilon_2=5.7056\times 10^{-5}$.
Indeed, this follows from Theorem 1.2 or (1.14) of \cite{BMOBR} for $t\geq 8\times 10^9$
and Theorem 1.9 of \cite{BMOBR} for $e^{21}\leq t<8\times 10^9$.
Hence, we obtain
\begin{equation}\label{eq23a}
\begin{split}
\sum_{p>e^{21}}\frac{1}{p^2}< & ~
\frac{(1-\epsilon_1)}{21e^{21}}+(1+\epsilon_1)\int_{e^{21}}^\infty t\left(\frac{1}{t^2\log t}\right)^\prime dt \\
= & ~ \frac{2\epsilon_1}{21e^{21}}+(1+\epsilon_1)\Ei(-21)<3.45369\times 10^{-11}
\end{split}
\end{equation}
and similarly
\begin{equation}\label{eq23b}
\sum_{\substack{p>e^{21}, \\ p\equiv 1\mathmod{6}}}\frac{1}{p^2}
<\frac{2\epsilon_2}{21e^{21}}+(0.5+\epsilon_2)\Ei(-21)<1.72721\times 10^{-11}.
\end{equation}
Similarly, we have
\begin{equation}\label{eq23c}
\sum_{p>e^{21}}\frac{1}{p^{4/3}}<\frac{2\epsilon_1}{21e^7}+(1+\epsilon_1)\Ei(-7)<1.15491\times 10^{-4}
\end{equation}
and
\begin{equation}\label{eq23d}
\sum_{\substack{p>e^{21}, \\ p\equiv 1\mathmod{6}}}\frac{1}{p^{4/3}}
<\frac{2\epsilon_2}{21e^7}+(0.5+\epsilon_2)\Ei(-7)<5.775241\times 10^{-5}.
\end{equation}

Clearly we have $H(0)=\prod_p (1+h(p)+h(p^2)+h(p^3)+h(p^4))$.
Calculation gives that
$\prod_{p\leq e^{21}} (1+h(p)+h(p^2)+h(p^3)+h(p^4))=0.18115182694\cdots$.
On the other hand, we have
\begin{equation}
\begin{split}
& \log \prod_{p>e^{21}}(1+h(p)+h(p^2)+h(p^3)+h(p^4)) \\
= & ~ 
\sum_{\substack{p>e^{21}, \\ p\equiv 5\mathmod{6}}}\log \left(1+\frac{1}{p(p-1)}-\frac{2p-1}{p^2(p-1)}\right) \\
& ~ +\sum_{\substack{p>e^{21}, \\ p\equiv 1\mathmod{6}}}
\log \left(1+\frac{9}{p(p-3)}-\frac{6p+9}{p^2(p-3)}+\frac{8p+3}{p^3(p-3)}-\frac{3}{p^3(p-3)}\right) \\
< & ~ \sum_{\substack{p>e^{21}, \\ p\equiv 5\mathmod{6}}}\frac{1}{p(p-1)}
+\sum_{\substack{p>e^{21}, \\ p\equiv 1\mathmod{6}}} \frac{9}{p(p-3)} \\
< & ~ \sum_{p>e^{21}}\frac{1.001}{p^2}+\sum_{\substack{p>e^{21}, \\ p\equiv 1\mathmod{6}}} \frac{8}{p^2}
<1.72749\times 10^{-10}
\end{split}
\end{equation}
with the aid of \eqref{eq23a} and \eqref{eq23b} and therefore
\begin{equation}
0.18115182694<H(0)<0.181151827.
\end{equation}

Similarly, we have $J(-1/3)=
\prod_p (1+\abs{h(p)}p^{1/3}+\abs{h(p^2)}p^{2/3}+\abs{h(p^3)}p+\abs{h(p^4)}p^{4/3})$
and $\prod_{p\leq e^{21}} (1+\abs{h(p)}p^{1/3}+\cdots +\abs{h(p^4)}p^{4/3})<114.07025$.
With the aid of \eqref{eq23c} and \eqref{eq23d}, we have $J(-1/3)<114.123$.

Putting $W(s)=H^\prime(s)/H(s)$, we have
\begin{equation}\label{eq24}
H^\prime(s)=W(s)H(s), H^{\prime\prime}(s)=(W^\prime(s)+W^2(s))H(s)
\end{equation}
for $s>-1/2$.
We observe that for $s>0$,
\begin{equation}
\begin{split}
W(s)= & ~ \frac{G^\prime(s)}{G(s)}-\frac{2\zeta^\prime(s+1)}{\zeta(s+1)}-\frac{L^\prime(s+1, \chi)}{L(s+1, \chi)} \\
= & ~ \sum_p \frac{2\log p}{p^{s+1}-1}+\frac{\chi(p)\log p}{p^{s+1}-\chi(p)}-\frac{\rho(p)\log p}{(p-\rho(p))p^s+\rho(p)}
\end{split}
\end{equation}
and
\begin{equation}
\begin{split}
W^\prime(s)=\sum_p & \left(\frac{p^s \rho(p)(p-\rho(p))\log^2 p}{((p-\rho(p))p^s+\rho(p))^2}\right. \\
& \left.-p^{s+1}\log^2 p\left(\frac{2}{(p^{s+1}-1)^2}+\frac{\chi(p)}{(p^{s+1}-\chi(p))^2}\right)\right).
\end{split}
\end{equation}
Proceeding as above, we have $W(+0)=2.8981415\cdots$ and $W^\prime(+0)=-9.06991\cdots$.
From \eqref{eq24}, we can find approximate values for $H^\prime(0)=W(+0)H(0)$ and
$H^{\prime\prime}(0)=(W^\prime(+0)+W^2(+0))H(0)$
as desired.
\end{proof}

Now we show our approximate formula for $M_g(x)$.
\begin{theorem}\label{th21}
Let
\begin{equation}
\begin{split}
C_0= & ~ b_0 H(0)=\frac{L(1, \chi)H(0)}{2}=0.05476217\cdots, \\
C_1= & ~ (2\gamma L(1, \chi)+L^\prime(1, \chi))H(0)+L(1, \chi)H^\prime(0)=0.4841912\cdots, \\
C_2= & \frac{L(1, \chi)H^{\prime\prime}(0)}{2}+(2\gamma L(1, \chi)+L^\prime(1, \chi))H^\prime(0) \\
& +\left(C_0 L(1, \chi)+2\gamma L^\prime(1, \chi)+\frac{L^{\prime\prime}(1, \chi)}{2}\right)H(0) \\
> & 0.5367018\cdots.
\end{split}
\end{equation}
Then we have $M_g(x)=C_0 \log^2 x+C_1\log x+C_2+O^*(4444x^{-1/3})$ for $x\geq 1$.
Moreover, we have $M_g(e^{21})/21^2=B_1=0.0790359\cdots $
and $M_g(x)\geq K_i \log^2 x$ for $x\leq t_i$,
where $K_i$ and $t_i$ $(i=1, 2, \ldots)$ are constants given in a file of \cite{Ymd21}.
\end{theorem}

\begin{proof}
Lemma \ref{lm23} yields that
\begin{equation}
\begin{split}
M_g(x) & ~ =\sum_d h(d)\sum_{m<x/d}\frac{f(m)}{m} \\
= & ~ \sum_d h(d)\left(b_0 \log^2\frac{x}{d}+b_1\log\frac{x}{d}+b_2+O^*\left(\frac{c_{1/3}}{(x/d)^{1/3}}\right)\right) \\
= & ~ (b_0\log^2 x+b_1\log x+b_2)H(0)+(-2b_0\log x+b_1)H^\prime(0)+b_0 H^{\prime\prime}(0) \\
& ~ +O^*\left(\frac{c_{1/3} J(-1/3)}{x^{1/3}}\right),
\end{split}
\end{equation}
where we note that Lemma \ref{lm23} holds even when $0<d<1$.
Using Lemma \ref{lm24}, we have $J(-1/3)<114.1231$ and obtain the desired formula.
Approximate values for $C_0, C_1, C_2$ can be found using Lemma \ref{lm24}.

The remaining part of the theorem can be confirmed by calculation,
which completes the proof of Theorem \ref{th21}.
\end{proof}

We also need the estimate for the remainder term $R_w$.

\begin{lemma}\label{lm25}
For $w\geq 1$, we have
\begin{equation}
\abs{\sum_{d_1, d_2\mid P(w)} \lm(d_1)\lm(d_2)r(\LCM[d_1, d_2])}
\leq\frac{(B_2 w+B_3 w^{1/2})^2}{M_g^2(w)},
\end{equation}
where
$$B_2=\frac{1}{\zeta(2)}\prod_p \left(1+\frac{2\rho(p)}{(p-\rho(p))(p+1)}\right)<0.883697$$
and
$$B_3=\frac{1}{2}\prod_p \left(1+\frac{4\rho(p)}{(p-\rho(p))p^{1/2}}\right)
+2\prod_p \left(1+\frac{2\rho(p)}{(p-\rho(p))p^{1/2}}\right)<12.8955.$$
\end{lemma}

\begin{proof}
We proceed as in the proof of Lemma 4.2 of \cite{RSS} to obtain
\begin{equation}\label{eq25}
\sum_{d\mid P(w)} \abs{\lm(d)\rho(d)}\leq
\frac{1}{M_g(w)}\sum_{q\leq w}\frac{\mu^2(q)\sigma_\rho(q)}{\varphi_\rho(q)},
\end{equation}
where $\sigma_\rho(q)=\prod_{p\mid q}(p+\rho(p))$ and $\varphi_\rho(q)=\prod_{p\mid q}(p-\rho(p))$
for squarefree $q$.
We see that
\begin{equation}
\sum_{q\leq w}\frac{\mu^2(q)\sigma_\rho(q)}{\varphi_\rho(q)}=\sum_{q\leq w}\mu^2(q)\sum_{b\mid q}h(b)
=\sum_b \mu^2(b)h(b)\sum_{q\leq w, b\mid q}\mu^2(q),
\end{equation}
where $h(b)=\prod_{p\mid b}2\rho(p)/(p-\rho(p))$.
Proceeding as in the proof of Lemma 4.2 of \cite{RSS} again, we have
\begin{equation}\label{eq26}
\begin{split}
\sum_{q\leq w}\frac{\mu^2(q)\sigma_\rho(q)}{\varphi_\rho(q)}
\leq ~ & \frac{w}{\zeta(2)}\sum_b\frac{\mu^2(b)h(b)}{\sigma(b)}
+\sqrt{w}\left(\frac{1}{2}\sum_b\frac{h(b)2^{\omega(b)}}{\sqrt{b}}+2\sum_b\frac{h(b)}{\sqrt{b}}\right) \\
= ~ & B_2 w+B_3 w^{1/2}.
\end{split}
\end{equation}
Combining \eqref{eq25} and \eqref{eq26}, we obtain
\begin{equation}
\sum_{d\mid w} \abs{\lm(d)\rho(d)}\leq \frac{B_2 w+B_3 w^{1/2}}{M_g(w)}
\end{equation}
and the lemma follows observing that
\begin{equation}
\abs{\sum_{d_1, d_2\mid P(w)} \lm(d_1)\lm(d_2)r(\LCM[d_1, d_2])}\leq\left(\sum_{d\mid P(w)} \lm(d)\rho(d)\right)^2.
\end{equation}
Numerical values of $B_2$ and $B_3$ can be obtained proceeding as in Lemma \ref{lm24}.
\end{proof}

\section{Bounding the smallest prime factor}

Now we shall show the following result, which is the most essential part in this paper.

\begin{theorem}
If $q_1, q_2, \ldots, q_{r-1}$ and $3$ are odd distinct primes and
$N=3^\alpha (q_1 q_2 \cdots q_{r-1})^2$ is $4$-perfect, then
$q_i\leq 162751$ for at least one $i\leq r-1$.
\end{theorem}

We write $\underline{p_3}(n)$ for the smallest prime factor of an integer $n$ congruent to $1\mathmod{3}$.
Assume that $N=3^\alpha (q_1 q_2 \cdots q_{r-1})^2$ is $4$-perfect and $q_i\geq 162779$ for all $i\leq r-1$.
Let $x\geq e^{12}$ and $w\geq e^7$ be real numbers satisfying $w=\sqrt{x}$
and $S(x)$ denote the set of integers $n\leq x$ in $S$ for a set $S$ of positive integers.
We divide the set $S_0=\{q_1, q_2, \ldots, q_{r-1}\}$ of prime divisors of $N$ other than three into two sets
$S_1$ and $S_2$, where $S_1=\{p: p\mid N, \underline{p_3}(p^2+p+1)>w\}$ and $S_2=S_0\backslash S_1$.

Hence, if $p\in S_1$, then $p\neq 3$ and $\sigma(p^2)=p^2+p+1$ has no prime divisor $\leq w$ other than three.
Clearly, $p^2\mid\mid N$ for any prime $p$ in $S_0=S_1\cup S_2$.

We set $A=A^{(x)\pm}=\{1\leq n\leq (x\mp 1)/6\}$, where we exclude multiples of two or three in order to
obtain better upper bounds for $\#S_0(x)$.
Moreover, we set $A_p=A_p^{(x)\pm}$ to be the solution set of the congruence
$(6n\pm 1)((6n\pm 1)^2+(6n\pm 1)+1)\equiv 0\mathmod{p}$ for any odd prime $p\geq 5$
and $A_2=A_3=\emptyset$.

Let $\rho(p)$ and $g(m)$ be the functions defined in Section \ref{LSEE}.
We immediately see that $\rho(p)=3$ for any prime $p\equiv 1\mathmod{3}$,
$\rho(p)=1$ for any prime $p\equiv 2\mathmod{3}$, and $\rho(2)=\rho(3)=0$.
It is clear that if $p=6n\pm 1$ is a prime greater than $w$ such that
$p^2+p+1$ has no prime divisor below $w$ other than three,
then $n$ must belong to $S(A, w)$.
Thus, we obtain
\begin{equation}
\# S_1(x)\leq 2w+\# S(A^{(x)+}, w)+\# S(A^{(x)-}, w)
\end{equation}
and Lemma \ref{lm0} gives
\begin{equation}\label{eq31}
\# S_1(x)\leq \frac{2((x+1)/6+w^2)}{M_g(w)}+\frac{2(B_2 w+B_3 w^{1/2})^2}{M_g^2(w)}
\end{equation}
with $M_g(x)$ satisfying the inequalities given in Theorem \ref{th21}.

For each $p\in S_2$, $p^2+p+1$ has a prime divisor $q\leq w$ other than three.
On the other hand, for each prime $q$ with $3<q\leq w$,
there exist at most two primes $p\in S_1$ such that $q$ divides $p^2+p+1$
as observed in the Introduction.
So that, $\# S_2(x)\leq 2\pi(w; 3, 1)\leq 2w$.
Hence, setting $w=\sqrt{x}$, we have
\begin{equation}
\# S_0(x)\leq 4\sqrt{x}
+\frac{\displaystyle (1+x^{-1})x}{3M_g(\sqrt{x})}
+\frac{\displaystyle 2(B_2+B_3 x^{-1/4})^2 x}{M_g^2(\sqrt{x})}.
\end{equation}
and partial summation gives
\begin{equation}\label{eq32}
\begin{split}
\sum_{p>x, p\in S_0}\frac{1}{p}\leq & -\frac{\#S_0(x)}{x}+\int_x^\infty \frac{\#S_0(t)}{t^2} dx \\
< & \frac{8}{\sqrt{x}}
+\int_x^\infty \frac{1+t^{-1}}{3M_g(\sqrt{t})}
+\frac{2(B_2+B_3 t^{-1/4})^2}{M_g^2(\sqrt{t})}\frac{dt}{t}.
\end{split}
\end{equation}
We observe that if there exists a function $\eta(t)$ such that
$M_g(t)>\eta(\log t)$ for $e^{2a}\leq t\leq e^{2b}$, then
\begin{equation}
\int_{e^{2a}}^{e^{2b}} \frac{dt}{t M_g^v(\sqrt{t})}
<\int_{e^{2a}}^{e^{2b}} \frac{dt}{t\eta^v(\log(\sqrt{t}))}=2\int_a^b \frac{du}{\eta^v(u)},
\end{equation}
where we put $u=\log(\sqrt{t})$ and $v\in \{1, 2\}$.
Hence, using Theorem \ref{th21}, the last integral in \eqref{eq32} is at most
\begin{equation}\label{eq33}
\begin{split}
& \int_{22.1}^\infty \frac{2(1+e^{-44.2})}{3F(u)}
+\frac{4(B_2+B_3 e^{-u/2})^2)}{F^2(u)}du \\
& +\int_{21}^{22.1} \frac{2(1+e^{-2u})}{3\times 21^2 B_1}+\frac{4(B_2+B_3 e^{-u/2})^2}{21^4 B_1^2} du \\
& +\int_{\log w}^{21} \frac{2(1+e^{-2u})}{M_g(e^u)}+
\frac{4(B_2+B_3 e^{-u/2})^2}{M_g^2(e^u)} du,
\end{split}
\end{equation}
where $F(u)=C_0 u^2+C_1 u+C_2-4444e^{-u/3}$.

We put $x=e^{18.6}$, which yields that $w=e^{9.3}$ and $\log w=9.3$.
We calculate
\begin{equation}
\int_{22.1}^\infty \frac{2(1+e^{-44.2})}{3F(u)}+\frac{4(B_2+B_3 e^{-u/2})^2)}{F^2(u)}du<0.4835711.
\end{equation}
Clearly, we have
\begin{equation}
\begin{split}
& \int_{21}^{22.1} \frac{2(1+e^{-2u})}{3\times 21^2 B_1}+\frac{4(B_2+B_3 e^{-u/2})^2}{21^4 B_1^2} du \\
< ~ & \frac{1.6(1+e^{-42})}{1323B_1}+\frac{4(B_2+B_3e^{-10.5})^2}{194481B_1^2}
<0.0230983.
\end{split}
\end{equation}
Finally, we confirmed by calculation with the aid of Theorem \ref{th21} that
\begin{equation}
\int_{9.3}^{21} \frac{2(1+e^{-2u})}{M_g(e^u)}+\frac{4(B_2+B_3 e^{-u/2})^2}{M_g^2(e^u)} du<0.461243.
\end{equation}

Inserting these estimates into \eqref{eq33}, we have
\begin{equation}
\int_x^\infty\frac{1+t^{-1}}{3M_g(\sqrt{t})}+\frac{2(B_2+B_3 t^{-1/4})^2}{M_g^2(\sqrt{t})}\frac{dt}{t}<0.9679124
\end{equation}
and \eqref{eq32} gives that
\begin{equation}\label{eq34}
\sum_{p>e^{18.6}, p\in S_0}\log\frac{p}{p-1}<0.968644.
\end{equation}

If $p\in S_0$, then $p>e^{12}$, $p_1=\underline{p_3}(p^2+p+1)>e^{12}$ and $\underline{p_3}(p_1^2+p_1+1)>e^{12}$
since $N$ has no prime factor $p$ with $5\leq p\leq e^{12}$ by our assumption.
We calculated the sum of reciprocals of such primes below $e^{18.6}$ to obtain
\begin{equation}
\sum_{p\leq e^{18.6}, p\in S_0}\frac{1}{p}<0.0109344.
\end{equation}
and, since $N$ has no prime factor $p$ with $5\leq p\leq e^{12}$,
\begin{equation}
\sum_{p\leq e^{18.6}, p\in S_0}\log\frac{p}{p-1}<
\frac{e^{12}}{e^{12}-1}\sum_{p\leq e^{18.6}, p\in S_0}\frac{1}{p}<0.010935.
\end{equation}
Combined with \eqref{eq34}, we obtain
\begin{equation}
\sum_{p\in S_0}\log\frac{p}{p-1}<0.979579<\log\frac{8}{3}.
\end{equation}
Thus, we conclude that
\begin{equation}
\frac{\sigma(N)}{N}<\frac{3}{2}\prod_{p\in S_0}\frac{p}{p-1}<4,
\end{equation}
which is a contradiction.
This proves the theorem.

\section{Proof of Theorem \ref{th1}}

We assume that $N=3^\alpha (q_1 q_2 \cdots q_{r-1})^2$ with $3, q_1, q_2, \ldots, q_{r-1}$ distinct odd primes
is an odd $4$-perfect number.
We write $p\rightarrow q$ when $q=\underline{p_3}(p^2+p+1)$.
Clearly, if $p$ divides $N$ and $p\rightarrow q$, then $q$ must also divide $N$.

We begin by proving that $7$ cannot divide $N$.
If $7$ divides $N$, then
$7^3\mid \sigma(79^2)\sigma(331^2)\mid N$ since
$7\rightarrow 19\rightarrow 127\rightarrow 5419\rightarrow 31\rightarrow 331\rightarrow 7$,
$\sigma(5419^2)=3\times 31\times 313\times 1009$, $313\rightarrow 181\rightarrow 79$,
and $\sigma(79^2)=3\times 7^2\times 43$.
Clearly, this contradicts the assumption that $N/3^\alpha$ is cubefree.

Nextly, we see that neither $11$ nor $13$ divides $N$.
If $11$ divides $N$, then $7\mid \sigma(11^2)\mid N$, which is impossible as we have observed above.
If $13$ divides $N$, then
$13\rightarrow 61$,
$\sigma(61^2)=3\times 13\times 97$,
$97\rightarrow 3169\rightarrow 3348577\rightarrow 181\rightarrow 79\rightarrow 7$,
and therefore $7\mid N$, which is impossible again.

In a similar way, we checked that no prime $p$ with $17\leq p\leq 162751$ can be the smallest prime factor of $N$.
Indeed, computer search revealed that the assumption that a prime $p$ with $11\leq p\leq 162751$ divides $N$
eventually yields that a prime $q<p$ must divide $N$, which is a contradiction.
For each prime $p$ with $11\leq p\leq 162751$,
we confirmed that $p\rightarrow p_1\rightarrow \cdots \rightarrow p_k$
for primes $p_1, p_2, \ldots, p_k$ all greater than $3$ such that
$p_i<2^{128}$ or $p_i\rightarrow p_{i+1}$ with $p_{i+1}<2^{32}$ for $i=1, \ldots, k-1$ and $p_k<p$ by PARI-GP.

Hence, no prime $p$ with $7\leq p\leq 162751$ can divide $N$.
So that, the smallest prime factor of $N$ other than three must be at least $162779>e^{12}$,
which is impossible as shown in the previous section.
This completes the proof of Theorem \ref{th1}.

\section{Proof of Theorem \ref{th0}}

In this section, we shall prove Theorem \ref{th0}.
Assume that an integer $N$ in the form \eqref{eq11} is $k$-perfect.
Since the case $k=2$ has been settled by \cite{Ymd3}, we assume that $k>2$.
Moreover, assume that $k=k_1 \ell^t$ for some integers $t\geq 0$ and $k_1\nequiv 0, 1\mathmod{\ell}$
if $2\beta+1=\ell^g$.

We begin by proving that $2\beta+1$ cannot be a power of $q_r$.
Indeed, if $2\beta+1=q_r^g$ for some integer $g>0$, then,
since $\sigma(q_r^\alpha)=\sigma(\ell^\alpha)\equiv 1\mathmod{\ell}$
and any prime factor of $\sigma(p^{2\beta})$ other than $q_r$ must be congruent to $1\mathmod{q_r}$,
we see that $\sigma(N)=q_r^{\alpha_1}M$ for certain integers $\alpha_1\geq \alpha$ and $M\equiv 1\mathmod{q_r}$.
On the other hand, we have $q_i^{2\beta}\equiv 1\mathmod{q_r}$ for each $i$
and therefore $N/q_r^\alpha\equiv 1\mathmod{q_r}$.
Thus, $k=k_1 q_r^t$ for integers $t\geq 0$ and $k_1\equiv 1\mathmod{q_r}$,
which contradicts our assumption.

Now we can take a prime factor $\ell\neq q_r$ of $2\beta+1$ since $2\beta+1$ cannot be a power of $q_r$.
We write $k=k_1 \ell^t$ with integers $t\geq 0$ and $k_1\nequiv 0\mathmod{\ell}$.
Let $i_0$ be the index for $q_{i_0}=\ell$ if it exists and
\begin{equation*}
\begin{split}
S= & \{i\leq r-1:q_i\equiv 1\mathmod{\ell}\}, \\
T= & \{i\leq r-1:q_i\nequiv 1\mathmod{\ell}, i\neq i_0, q\mid \sigma(q_i^{2\beta})\text{ for some }q\neq q_r\}, \\
U= & \{i\leq r-1:q_i\nequiv 1\mathmod{\ell}, i\neq i_0, q\nmid\sigma(q_i^{2\beta})\text{ for any }q\neq q_r\}. \\
\end{split}
\end{equation*}
Hence, we can write $\{i: 1\leq i\leq r-1\}=S\cup T\cup U\cup \{i_0\}$.
Moreover, $\Omega_0(k)$ denotes the number of odd distinct prime divisors of $k$
and $\Omega_1(k)$ denotes the number of prime divisors of $k$ congruent to $1\mathmod{\ell}$,
both counted with multiplicity.

Like Lemmas 3.1 and 3.2 of \cite{Ymd1}, we see that $\# S\leq 2\beta+t$ and $\# T\leq 2\beta(2\beta+t)+\Omega_1(k)$
respectively.
If $i\in U$, then $\sigma(q_i^{2\beta})=q_r^{\xi_i}$ for some integer $\xi_i>0$.
By the theorem of Bang\cite{Ban} or Zsigmondy\cite{Zsi}, $2\beta+1$ must be prime.
Proceeding as in Lemma 3.3 of \cite{Ymd1}, we have
\begin{equation}
2(2\beta)^{\# U}-1\leq \omega(kN)\leq \# U+\omega(k)+\# S+\# T+2.
\end{equation}
Since $u\leq 2^{u-1}$ for $u\geq 3$, we have $\# U\leq 2$ or
\begin{equation}
\frac{3}{2}(2\beta)^{\# U}\leq \omega(k)+\# S+\# T+3\leq \omega(k)+(2\beta+1)(2\beta+t)+\Omega_1(k)+3.
\end{equation}
Noting that $\Omega_1(k)+t=\Omega_0(k)$, we have
\begin{equation}
\# U\leq \max\left\{2,
\frac{\frac{2}{3}(\omega(k)+(2\beta+1)(2\beta+\Omega_0(k)))+2}{\log(2\beta)}\right\}.
\end{equation}

We can easily see that $\Omega_0(k)\leq (\log k)/(\log 3)$ and $\omega(k)\leq \log k+2-\log 6<\log k+0.20825$.
Now we show that
\begin{equation}\label{eq51}
\log k<\log\log r+0.24351.
\end{equation}
Indeed, if $k\geq 8$, then, with the aid of Theorem 5.9 and Proposition 5.15 in \cite{Dus}, we have
\begin{equation}
k<\prod_{i=2}^{r+1}\frac{P_i}{P_i-1}<0.9\log P_{r+1}<1.17\log r,
\end{equation}
where $P_i$ denotes the $i$-th prime, and therefore $\log k<\log\log r+0.24$, implying \eqref{eq51}
(but Ramar\'{e}'s zero density estimate in \cite{Ram2}, on which Dusart's estimates in \cite{Dus} are based,
is objected by \cite{BKLNW}.
Corollary 11.2 in \cite{BKLNW} can also be used to obtain Dusart's estimates.
In practice, Theorems 3 and 8 in \cite{RS} suffice for our purpose).
If $3\leq k\leq 7$, then $r\geq r_0$ with
$(k, r_0)=(3, 12)$ by Hagis \cite{Hag}, Kishore \cite{Kis}, or Reidlinger \cite{Rei}
and $(4, 23)$, $(5, 56)$, $(6, 142)$, and $(7, 373)$ by Nakamura \cite{Nak}.
Hence, we obtain $\log k<\log\log r+\log 4-\log\log 23<\log\log r+0.24351$ and \eqref{eq51} holds.

Now we have $\Omega_0(k)<(\log\log r+0.24351)/\log 3$ and $\omega(k)<\log\log r+0.4518$.
So that, we have
\begin{equation}
\# S+\# T<(2\beta+1)\left(2\beta+\frac{\log\log r+0.24351}{\log 3}\right).
\end{equation}
Moreover, $\# U\leq 2$ or
\begin{equation}\label{eq52}
\# U<\frac{\displaystyle\frac{2}{3}\left((2\beta+1)\left(2\beta+\frac{\log\log r+0.24351}{\log 3}\right)+\log\log r\right)+2.3012}{\log(2\beta)}.
\end{equation}
Thus, we have
\begin{equation}\label{eq53}
\begin{split}
& r<2+(2\beta+1)\left(2\beta+\frac{\log\log r+0.24351}{\log 3}\right)+ \\
& \max\left\{2, \frac{\displaystyle\frac{2}{3}\left((2\beta+1)\left(2\beta+\frac{\log\log r+0.24351}{\log 3}\right)+\log\log r\right)+2.3012}{\log(2\beta)}\right\}.
\end{split}
\end{equation}
For $\beta=1$, $2$, $3$, $4$, $5$, $6$, and $7$, this implies that $r\leq 14$, $30$, $56$, $90$, $132$, $182$,
and $240$ respectively.

If $\beta\geq 8$, then, setting $\delta$ to be the number of indices $i$ in $T$ such that
$\sigma(q_i^{2\beta})=p q_r^\xi$ for some prime $p\neq q_r$ and some exponent $\xi$
(we note that $p=q_j$ for some $j\in S$ since $i\in T$)).
Then, proceeding as in the proof of Lemma 3.2 of \cite{Ymd2}, we have
$\delta+2(\# T-\delta)\leq 2\beta(2\beta+t)+\Omega_1(k)$ or, in other words,
$\# T\leq \beta(2\beta+t)+(\delta+\Omega_1(k))/2$.
From \cite{Ymd3}, we can deduce that $\delta\leq 4\# S\leq 4(2\beta+t)$ and
therefore $\# T\leq (\beta+2)(2\beta+t)+\Omega_1(k)/2$.
Thus, instead of \eqref{eq52} and \eqref{eq53}, we have
\begin{equation}\label{eq54}
\# U<\frac{\displaystyle\frac{2}{3}\left((\beta+3)\left(2\beta+\frac{\log\log r+0.24351}{\log 3}\right)+\log\log r\right)+2.3012}{\log(2\beta)}
\end{equation}
and
\begin{equation}\label{eq55}
\begin{split}
& r<2+(\beta+3)\left(2\beta+\frac{\log\log r+0.24351}{\log 3}\right)+ \\
& \max\left\{2, \frac{\displaystyle\frac{2}{3}\left((\beta+3)\left(2\beta+\frac{\log\log r+0.24351}{\log 3}\right)+\log\log r\right)+2.3012}{\log(2\beta)}\right\}
\end{split}
\end{equation}
respectively.
This implies that $r<\beta^e$ and we can deduce from \eqref{eq54} that $\# U\leq 2$.
Hence, using \eqref{eq55} again, we conclude that
\begin{equation}
r<\left(2\beta+\frac{\log\log \beta+1.24351}{\log 3}\right)(2\beta+1)+4,
\end{equation}
as we desired.
This proves Theorem \ref{th0}.

{}
\end{document}